\documentclass[reqno]{amsart}

\usepackage[dvipsnames]{xcolor}
\usepackage{amsmath}
\usepackage{enumerate}
\usepackage{enumitem}
\usepackage{amssymb}
\usepackage{extarrows}
\usepackage{amsthm}
\usepackage{thmtools}
\usepackage{mathrsfs}
\usepackage[colorlinks, linkcolor = SeaGreen, citecolor = SeaGreen]{hyperref}
\usepackage{bm}
\hypersetup{colorlinks = true}

\newtheorem{thm}{Theorem}[section]

\newtheorem{lem}[thm]{Lemma}

\newtheorem{ques}[thm]{Question}
\theoremstyle{definition}

\theoremstyle{remark}
\newtheorem{remark}[thm]{Remark}

\newcommand{\D}{\mathbb D}

\newcommand{\R}{\mathbb{R}}

\newcommand{\Hol}{\mathrm{Hol}}

\title[Commuting analytic and co-analytic Toeplitz operators on $L_h^2$]
{Commuting analytic and co-analytic Toeplitz operators
on the harmonic Bergman space}

\author[Tian]{Xianchi Tian}
\address{College of Mathematics and Statistics, Chongqing University, Chongqing, 401331, P. R. China}
\email{202306021073t@stu.cqu.edu.cn}

\author[Wang]{Jiawei Wang}
\address{School of Mathematical Sciences, Inner Mongolia University, Hohhot, 010021, P. R. China}
\email{jiaweiwang1516@163.com}

\author[Zhao]{Xianfeng Zhao}
\address{College of Mathematics and Statistics, Chongqing University, Chongqing, 401331, P. R. China}
\email{xianfengzhao@cqu.edu.cn}


\makeatletter

\newcommand{\Rmnum}[1]{\expandafter\@slowromancap\romannumeral #1@}
\@namedef{subjclassname@2020}{\textup{2020} Mathematics Subject Classification}
\makeatother

\begin{document}

\begin{abstract}
By constructing an appropriate quadratic form, we prove that if an analytic Toeplitz operator commutes with a co-analytic Toeplitz operator on the harmonic Bergman space, then at least one of the symbols is constant. This resolves an open question posed by Choe and Lee in 1999.
\end{abstract}

\maketitle
\section{Introduction}\label{S1}
Let $dA(z)=\frac{1}{\pi}dxdy$ denote the normalized Lebesgue area measure on the open unit disk $\mathbb D$ in the complex plane $\mathbb C$. $L^2(\mathbb D, dA)$ is the
Hilbert space of Lebesgue square integrable  functions on $\mathbb D$ with the inner product
$$\langle f, g\rangle=\int_{\mathbb D}f(z)\overline{g(z)}dA(z).$$
The Bergman space $L_a^2$ is the closed subspace of $L^2(\mathbb D, dA)$ consisting of analytic functions on $\mathbb D$. It is well known that $L_a^2$ is a  reproducing kernel Hilbert space. The harmonic Bergman space $L_h^2$ is the closed subspace of $L^2(\mathbb D, dA)$ consisting of all complex-valued harmonic functions on $\mathbb D$. It is easy to verify  the following decomposition for $L_h^2$:
$$L_h^2=zL_a^2\oplus\overline{L_a^2},$$
where $\overline{L_a^2}$ denotes  the complex conjugate of $L_a^2$.
Therefore, \[\big\{\sqrt{n+1}z^n: n=1,2,\cdots\big\}
\cup\big\{\sqrt{n+1}\overline z^{n}:n=0, 1, 2,\cdots\big\}
\]
is an orthonormal basis for $L_h^2$. Moreover, the harmonic Bergman space is also a reproducing kernel Hilbert space and the reproducing kernel is given by
 $$R_z(w)=\frac{1}{(1-\overline{z}w)^2}+\frac{1}{(1-z\overline{w})^2}-1, \ \ \ \  w\in \mathbb D.$$

Let $P$ and $Q$ be the orthogonal projections from $L^2(\mathbb D, dA)$ onto $L_a^2$ and $L_h^2$, respectively. It was shown in
\cite{CL} that $P$ and $Q$ have the following relationship:
\begin{align*}\label{P and Q}
Q(f)=P(f)+\overline{P(\overline{f})}-P(f)(0), \ \ \ \ \ \ \  f\in L^2(\mathbb D, dA).
\end{align*}

 For $\varphi\in L^2(\mathbb D, dA)$, the Toeplitz operator $T_\varphi$ with symbol $\varphi$ on  the harmonic Bergman space $L_h^2$ is densely defined by
$$T_\varphi f=Q(\varphi f)$$
for harmonic polynomials $f$. In particular, a Toeplitz operator is called analytic  (co-analytic) if its symbol $\varphi$ is analytic  (conjugate analytic)  on $\mathbb D$. For more details concerning Toeplitz operators on the Bergman and harmonic Bergman spaces, one can consult \cite{M}, \cite{GZ}, \cite{CL2}, \cite{Zhu} and \cite{LL}.

Some properties of Toeplitz operators on the harmonic Bergman space are analogous to those on the Bergman space. For example, in the case of bounded harmonic symbols, there exist no nonzero compact Toeplitz operators; and in the case of continuous symbols on the closed disk, the essential spectrum of the Toeplitz operator equals the image of the unit circle under the symbol (see \cite{GZ}). Nevertheless, since the properties of the harmonic Bergman space differ from those of the Bergman space, the Toeplitz operators on these spaces also exhibit significant differences. It is well-known that the spectrum of any analytic Toeplitz operator on the Bergman space is connected. However, there exists  a quadratic polynomial such that  the corresponding Toeplitz operator on the harmonic
Bergman space has a disconnected spectrum (see \cite{WZ}). Furthermore, the commutativity of two Toeplitz operators with bounded harmonic symbols on the Bergman space can be completely characterized by the conformally invariant mean value property of harmonic functions (\cite{AC}); however, this technique is not available in the setting of  the harmonic Bergman space.

The commutativity problem for Toeplitz operators on the harmonic Bergman space $L_h^2$ was first investigated by Ohno \cite{O}, who observed that, for an analytic symbol $f$, $T_f$ commutes with $T_z$ if and only if $f$ is a polynomial of degree at most $1$. Motivated by this work, Choe and Lee \cite{CL} considered the commutativity of densely defined Toeplitz operators on $L_h^2$ and established that (i) two  Toeplitz operators with symbols in $L_a^2$ commute on the harmonic Bergman space if and only if a nontrivial linear combination of their symbols is constant; and (ii) a Toeplitz operator with symbol in $L_h^2$ commutes with a Toeplitz operator having a harmonic polynomial symbol on the harmonic Bergman space if and only if a nontrivial linear combination of those two symbols is constant. One may naturally ask whether the conclusion in (ii)
is true for general harmonic symbols. As pointed out in  \cite{CL}, this problem can be reduced to the following special
case:
\begin{ques}\label{Q}
\emph{(Choe-Lee, 1999)} Let $f\in L_a^2$ and $g\in L_a^2$. Does it
follow that either $f$ or $g$ is constant if $T_{f}$ and $T_{\overline{g}}$
commute on the harmonic Bergman space $L_h^2$?
\end{ques}

Five years later, Choe and Lee  \cite{CL1} proved that the answer to the above question is affirmative whenever $f$ and $g$ are bounded analytic on $\mathbb D$ (i.e., $f, g\in H^\infty$)  and one of them is noncyclic for the multiplication operator $M_{\overline{z}}$ on the Bergman space.  Following some ideas in \cite{CL} and \cite{CL1}, Ding (\cite{D}) asserted that if $f$ and $g$ are in the Bergman space and one of them is bounded, then $T_f$ and $T_{\overline{g}}$ commute on $L_h^2$ if and only if either $f$ or $g$ is constant.  However, his proof appears to contain some gaps, as it relies on a misapplication of the Axler-Shields theorem. Therefore, Question \ref{Q} remains open even when both symbols are bounded.

For further results concerning the commutativity of analytic and co-analytic Toeplitz operators on the harmonic Bergman space, we refer to \cite{KQT}, where several conclusions from \cite{CL} were extended to the standard weighted harmonic Bergman space. The same commutativity problem has also been investigated for other classes of symbols. In particular, using the Mellin transform, the authors of \cite{DZ} and \cite{LRZ} studied Toeplitz operators with certain quasi-homogeneous symbols on the harmonic Bergman space. More recently, Iqtaish, Louhichi, and Yousef \cite{ILY} obtained a complete characterization of the bounded Toeplitz operators $T_f$
on the harmonic Bergman space that commute with $T_{z+\overline{g}}$, under the assumption that the symbol $f$ has a polar decomposition truncated above and that $g\in H^\infty$.

Motivated by the work of Choe and Lee \cite{CL,CL1} and Ding \cite{D}, this paper is devoted to a further study of the commutativity problem for an analytic Toeplitz  operator and a co-analytic Toeplitz operator on the harmonic Bergman space.  By constructing an appropriate quadratic form, we give an affirmative answer to Question \ref{Q} using entirely elementary techniques. More precisely, we will prove the following theorem in Section \ref{S3}:
\begin{thm}\label{thm:main}
Suppose that  $f$ and $g$ are in $L_a^2$. Then
\[
T_fT_{\overline g}=T_{\overline g}T_f
\]
on the harmonic Bergman space $L_h^2$ if and only if either $f$ or $g$ is constant.
\end{thm}

Furthermore, concerning properties such as boundedness, positivity, invertibility, compactness, and Schatten classes of Toeplitz operators on the harmonic Bergman space, we refer to \cite{M}, \cite{SZ} and \cite{WZ1}.

\section{Preliminaries}\label{S2}

This section collects basic properties of the orthogonal projection $Q: L^2(\D, dA)\rightarrow L_h^2$  that will be used frequently in the proof of the main result.

\begin{lem}\label{lem:projection}
If $p$ and $q$ are nonnegative integers, then
\begin{align*}
  Q(z^p \overline{z}^q)=
  \begin{cases}
    \frac{p-q+1}{p+1}z^{p-q}, \ \ &p\geqslant q, \vspace{2mm}\\[3pt]
    \frac{q-p+1}{q+1}  \overline{z}^{q-p}, \ \ &p<q.
  \end{cases}
\end{align*}
\end{lem}
\begin{proof}
From Section \ref{S1}, we recall the following relation between the orthogonal projections $Q$  and $P$:
$$Q(\varphi)=P(\varphi)+\overline{P(\overline{\varphi})}-P(\varphi)(0),$$
where  $\varphi\in L^2(\mathbb D, dA)$.
Then  combining the above  identity with the following well-known formula:
  \begin{align*}
     P(z^p\overline z^q)=
     \begin{cases}
    \frac{p-q+1}{p+1}z^{p-q},\ \ &p\geqslant q,\vspace{2mm}\\[3pt]
    0,\ \ &p<q,
    \end{cases}
  \end{align*}
we obtain the desired conclusion.
\end{proof}

Let $\Hol(\D)$ denote the space of analytic functions on $\D$, equipped with the topology of uniform convergence on compact subsets of $\D$. This topology is induced by the family of seminorms:
$$\|h\|_{K}=\sup_{z\in K}|h(z)|,\qquad h\in \Hol(\D),$$
where $K$ ranges over all compact subsets of $\D$. It is well-known that
we can extend the domain of the Bergman projection $P$ to $L^1(\mathbb D, dA)$:
$$(Pf)(z)=\int_{\D}\frac{f(w)}{(1-z\overline{w})^2} dA(w)$$
for $f\in L^1(\mathbb D, dA)$ and $z\in \mathbb D$.
Therefore, $Pf$ is analytic on $\mathbb D$ for each $f\in L^1(\mathbb D, dA)$.
Moreover, recall that $P: L^1(\D,dA)\to L_a^1(\D,dA)$ is unbounded.
Nevertheless, $P$ is a continuous map from $L^1(\mathbb D,dA)$  into $\Hol(\D)$, as stated in the following lemma.

\begin{lem}\label{Lbdd}
The Bergman projection $P$ is continuous from $L^1(\mathbb D,dA)$  into $\Hol(\D)$.
\end{lem}

\begin{proof}
Suppose $\{f_n\}_{n=1}^{\infty}\subset L^1(\D, dA)$ such that $\|f_n-f\|_{L^1(\D,dA)}\to 0$ as $n\rightarrow \infty$. It suffices to prove that $\lim\limits_{n\rightarrow \infty}\|Pf_n-Pf\|_K = 0$ for every compact subset $K\subset\mathbb D$.  Fix an arbitrary compact set $K\subset\mathbb D$. Then choose  $\rho\in (0, 1)$ such that $K\subset\{z\in\mathbb D:|z|\leqslant \rho\}$.
Clearly,
$$\left|\frac{1}{(1-z\overline{w})^2}\right| \leqslant  \frac{1}{(1-\rho)^2}$$
for each $z\in K$ and  $w\in\mathbb D$.
It follows that
\begin{align*}
\|Pf-Pf_n\|_{ K}&=\sup_{z\in K}|(Pf-Pf_n)(z)|\\
&=\sup_{z\in K}\left|\int_{\mathbb D}\frac{f_n(w)-f(w)}{(1-z\overline{w})^2}\,dA(w)\right|\\
&\leqslant \sup_{z\in K}\int_{\mathbb D}\frac{|f_n(w)-f(w)|}{|1-z\overline{w}|^2}\,dA(w)\\
&\leqslant \frac{1}{(1-\rho)^2}\int_{\mathbb D}|f_n(w)-f(w)|\,dA(w)\\
&=\frac{1}{(1-\rho)^2}\|f_n-f\|_{L^1(\mathbb D,dA)}.
\end{align*}
Since $\|f_n-f\|_{L^1(\mathbb D,dA)}\to 0$ as $n\rightarrow \infty$, we have $$\lim\limits_{n\rightarrow \infty}\|Pf_n-Pf\|_K= 0,$$ to get that $Pf_n\to Pf \ (n\rightarrow \infty)$ in $\Hol(\D)$. This completes the proof of the lemma.
\end{proof}
Using the relation between $Q$  and $P$, the domain of   $Q$ can be also  extended to $L^1(\mathbb D, dA)$:
\begin{align}\label{conjugate Q}
Qf=Pf+\overline{P(\overline{f})}-Pf(0),\ \ \ \   f\in L^1(\D, dA).
\end{align}
Then it follows from Lemma \ref{Lbdd} that $Q$ is continuous from $L^1(\D, dA)$ into $\Hol(\D)+\overline{\Hol(\D)}$, where $\overline{\Hol(\D)}$ denotes the complex conjugate of $\Hol(\D)$.

The following simple property of the orthogonal projection $Q$ will be used in the  proof of our main result.

\begin{lem}\label{Lconjugation}
For every $u\in L^1(\D, dA)$, we have
$$Q(\overline{u})=\overline{Q(u)}.$$ In particular, the  identity holds for $u\in L^2(\D, dA)$.
\end{lem}

\begin{proof}
 From (\ref{conjugate Q}), we have that
$$Q(\overline{u})=P(\overline{u})+\overline{P(u)}-P(\overline{u})(0)$$
and
$$\overline{Q(u)}=\overline{P(u)}+P(\overline{u})-\overline{P(u)(0)}.$$
Moreover, since
$$P(\overline{u})(0)=\int_{\mathbb D}\overline{u(w)}\,dA(w)=\overline{\int_{\mathbb D}u(w)\,dA(w)}=\overline{P(u)(0)},$$
we finally obtain  that
$$Q(\overline{u})=\overline{Q(u)},\ \ \ \  u\in L^1(\D, dA).$$
This completes the proof.
\end{proof}

\section{Proof of Theorem \ref{thm:main}}\label{S3}

This section is devoted to the proof of Theorem \ref{thm:main}.

\begin{proof}[Proof of Theorem \ref{thm:main}]
Clearly, we need only to show the necessity. Suppose that $f\in L_a^2$ and $g\in L_a^2$ satisfying $$ T_fT_{\overline g}=T_{\overline g}T_f$$ on the harmonic Bergman space $L_h^2$.
Without loss of generality, we may assume that $$f(0)=g(0)=0$$ and write
$$
f(z)=\sum_{j=1}^{\infty}a_jz^j,
\qquad
g(z)=\sum_{j=1}^{\infty}b_jz^j.$$
Since $f, g\in L_a^2$ , we have that
$$
\sum_{j=1}^{\infty}\frac{|a_j|^2}{j+1}<\infty \ \  \ \ \mathrm{and}\ \ \ \
\sum_{j=1}^{\infty}\frac{|b_j|^2}{j+1}<\infty.
$$

Fix a nonnegative integer $s$ and let $n\geqslant 1$. We first compare the coefficients of $z^{n+s}$ in
$T_fT_{\overline g}z^n$ with those in $T_{\overline g}T_fz^n$.
On the one hand, since the projection $Q$ is continuous from $L^1(\mathbb D, dA)$ into $\Hol(\D)+\overline{\Hol(\D)}$, we have by Lemma \ref{lem:projection} that
\begin{equation*}\label{eq:Tgbarn}
\begin{aligned}
T_{\overline g}z^n=Q(\overline gz^n)=\sum_{k=1}^{n}\frac{n-k+1}{n+1}\overline{b}_kz^{n-k}+\sum_{k=n+1}^{\infty}
\frac{k-n+1}{k+1}\overline{b}_k\,\overline z^{\,k-n}.
\end{aligned}
\end{equation*}
Denote $$R_n=\sum_{k=n+1}^{\infty}
\frac{k-n+1}{k+1}\overline{b}_k\,\overline z^{\,k-n}=\sum_{m=1}^\infty \frac{m+1}{m+n+1}\overline{b}_{m+n}\overline{z}^m.$$ Then
\begin{align*}
  \|R_n\|_{L^2(\D, dA)}^2&=\sum_{m=1}^\infty \bigg|\frac{m+1}{m+n+1}\bigg|^2\frac{|b_{m+n}|^2}{m+1}\\
  &=\sum_{m=1}^\infty\frac{m+1}{(m+n+1)^2}|b_{m+n}|^2\\
  &\leqslant \sum_{m=1}^\infty\frac{|b_{m+n}|^2}{m+n+1}\\
  &=\sum_{k=n+1}^{\infty}\frac{|b_{k}|^2}{k+1}<\infty,
\end{align*}
which implies that $R_n\in \overline{L_a^2}.$
Using Lemma \ref{lem:projection} again, we obtain
\begin{align}\label{TfTg}
\begin{split}
T_fT_{\overline g}z^n&=Q(f T_{\overline{g}}z^n)\\
&=Q\Big(f\sum_{k=1}^{n}\frac{n-k+1}{n+1}\overline{b}_kz^{n-k}\Big)+Q(fR_n)\\
&=\sum_{k=1}^{n}\frac{n-k+1}{n+1}\overline{b}_kQ(fz^{n-k})+\sum_{k=n+1}^{\infty}\frac{k-n+1}{k+1}\overline{b}_kQ(f\overline z^{\,k-n})\\
&=\sum_{k=1}^{n}\sum_{j=1}^{\infty}\frac{n-k+1}{n+1}\overline{b}_ka_jQ(z^{n+j-k})+\sum_{k=n+1}^{\infty}\sum_{j=1}^{\infty}\frac{k-n+1}{k+1}a_j\overline{b}_kQ(z^j\overline z^{\,k-n})\\
&=\sum_{k=1}^{n}\sum_{j=1}^{\infty}\frac{n-k+1}{n+1}\overline{b}_ka_jz^{n+j-k}+\sum_{k=n+1}^{\infty}\sum_{j=1}^{\infty}\frac{k-n+1}{k+1}a_j\overline{b}_kQ(z^j\overline z^{\,k-n}),
\end{split}
\end{align}
where the third equality is due to the continuity of $Q$  from $L^1(\D, dA)$ to $\Hol(\D)+\overline{\Hol(\D)}$, and the fourth equality follows from the boundedness of $Q$ on $L^2(\D, dA)$.

Denote
$$\text{I}=\sum_{k=1}^{n}\sum_{j=1}^{\infty}\frac{n-k+1}{n+1}\overline{b}_ka_jz^{n+j-k}$$
and
$$\text{II}=\sum_{k=n+1}^{\infty}\sum_{j=1}^{\infty}\frac{k-n+1}{k+1}a_j\overline{b}_kQ(z^j\overline z^{\,k-n}).$$
For $1\leqslant k\leqslant  n$, the relevant terms in I are analytic. To extract the coefficient of $z^{n+s}$ from the first sum, we substitute  $n+j-k=n+s,$ which yields $j=k+s$. Consequently, the coefficient of  $z^{n+s}$ in I is
$$\sum_{k=1}^{n}\frac{n-k+1}{n+1}a_{k+s}\overline{b}_k.$$
For $k\geqslant n+1$, we need only to consider the analytic part of II. In fact, we obtain by Lemma \ref{lem:projection} that the analytic part of II is given by

$$\sum_{k=n+1}^{\infty}\sum_{j=k-n}^{\infty}\frac{k-n+1}{k+1}a_j\overline{b}_k\frac{j-(k-n)+1}{j+1}z^{j-(k-n)}.
$$
Thus the  coefficient of $z^{n+s}$ in II is $$\sum_{k=n+1}^{\infty}\frac{k-n+1}{k+1}\frac{n+s+1}{k+s+1}a_{k+s}\overline{b}_k.$$
As a result, the coefficient of $z^{n+s}$ in
$T_fT_{\overline g}z^n$ is
\begin{align*}\label{eq:firstcoefficient}
A_{n+s}&=\sum_{k=1}^{n}\frac{n-k+1}{n+1}
a_{k+s}\overline{b}_k
+\sum_{k=n+1}^{\infty}
\frac{(n+s+1)(k-n+1)}{(k+1)(k+s+1)}a_{k+s}\overline{b}_k.
\end{align*}
Letting $$A_1=\sum_{k=1}^{n}\frac{n-k+1}{n+1}
a_{k+s}\overline{b}_k$$ and $$A_2=\sum_{k=n+1}^{\infty}
\frac{(n+s+1)(k-n+1)}{(k+1)(k+s+1)}a_{k+s}\overline{b}_k,$$
we have  $A_{n+s}=A_1+A_2$.

Using the same argument as the one in (\ref{TfTg}),  we obtain that
\begin{align*}
T_{\overline{g}}T_{f}z^n&=Q(\overline{g} T_{f}z^n)=Q(\overline{g}fz^n)\\
&=Q\Big(\sum_{k=1}^{\infty}\overline{b}_k\overline{z}^kfz^n\Big)\\
&=\sum_{k=1}^{\infty}\overline{b}_kQ\Big(\sum_{j=1}^{\infty}a_jz^{n+j}\overline{z}^{k}\Big)\\
&=\sum_{k=1}^{\infty}\sum_{j=1}^{\infty}a_j\overline{b}_kQ(z^{n+j} \overline{z}^{k}).
\end{align*}
Therefore, the analytic part of $T_{\overline{g}}T_fz^n$ is given by
$$\sum_{k=1}^{\infty}~\sum_{j=\max\{1,k-n\}}^{\infty}\frac{n+j-k+1}{n+j+1}a_j\overline{b}_kz^{n+j-k},$$
which yields that  the coefficient of $z^{n+s}$ in $T_{\overline{ g}}T_{f}z^n$ is
\begin{equation*}\label{eq:secondcoefficient}
B_{n+s}=\sum_{k=1}^{\infty}
\frac{n+s+1}{n+k+s+1}a_{k+s}\overline{b}_k.
\end{equation*}
Using the Cauchy-Schwarz inequality, we have
\begin{align*}
  \sum_{k=1}^{\infty}\bigg|\frac{n+s+1}{n+k+s+1}a_{k+s}\overline{b}_k\bigg|&\leqslant (n+s+1)\sum_{k=1}^{\infty}\frac{|a_{k+s}\overline{b}_k|}{k+s+1}\\
  &=(n+s+1)\sum_{k=1}^{\infty}\frac{|a_{k+s}|}{\sqrt{k+s+1}}\frac{|\overline{b}_k|}{\sqrt{k+s+1}}\\
  &\leqslant (n+s+1)\sum_{k=1}^{\infty}\frac{|a_{k+s}|}{\sqrt{k+s+1}}\frac{|\overline{b}_k|}{\sqrt{k+1}}\\
  &\leqslant  (n+s+1)\bigg(\sum_{k=1}^{\infty}\frac{|a_{k+s}|^2}{k+s+1}\bigg)^{\frac{1}{2}}\bigg(\sum_{k=1}^{\infty}\frac{|\overline{b}_k|^2}{k+1}\bigg)^{\frac{1}{2}}\\
  &\leqslant (n+s+1)\bigg(\sum_{j=1}^{\infty}\frac{|a_{j}|^2}{j+1}\bigg)^{\frac{1}{2}}\bigg(\sum_{k=1}^{\infty}\frac{|b_k|^2}{k+1}\bigg)^{\frac{1}{2}}\\
  &<\infty
\end{align*}
 for $ n\geqslant 1$ and $s\geqslant 0.$ Thus the series  $\sum\limits_{k=1}^{\infty}
\frac{n+s+1}{n+k+s+1}a_{k+s}\overline{b}_k$ is  absolutely convergent.
 Denote
 $$B_1=\sum_{k=1}^n\frac{n+s+1}{n+k+s+1}a_{k+s}\overline{b}_k$$
 and
 $$B_2=\sum_{k=n+1}^{\infty}
\frac{n+s+1}{n+k+s+1}a_{k+s}\overline{b}_k.$$
 Then we have $B_{n+s}=B_1+B_2$.
The relation $T_fT_{ \overline{g}}=T_{ \overline{g}}T_f$ implies that
$A_{n+s}=B_{n+s}$ for all $n\geqslant 1$ and $s\geqslant 0$. Consequently, we have $$(A_1-B_1)+(A_2-B_2)=0.$$ For
$1\leqslant k\leqslant n$, a direct calculation yields
\begin{equation*}\label{eq:difference1}
\frac{n-k+1}{n+1}-\frac{n+s+1}{n+k+s+1}
=-\frac{k(k+s)}{(n+1)(n+k+s+1)}.
\end{equation*}
For $k>n$, we similarly have
\begin{equation*}\label{eq:difference2}
\begin{aligned}
&\frac{k-n+1}{(k+1)(k+s+1)}
-\frac1{n+k+s+1}\\
&\qquad=-\frac{n(n+s)}
{(k+1)(k+s+1)(n+k+s+1)}.
\end{aligned}
\end{equation*}
Thus,   using $(A_1-B_1)+(A_2-B_2)=0$ we  obtain
\begin{align*}
  &\sum_{k=1}^n\frac{-k(k+s)(n+s+1)}{(n+1)(n+k+s+1)(n+s+1)}a_{k+s}\overline{b}_k\\
  &\ \ +\sum_{k=n+1}^{\infty}\frac{-n(n+s)(n+s+1)}
{(k+1)(k+s+1)(n+k+s+1)}a_{k+s}\overline{b}_k=0,
\end{align*}
which is equivalent to
\begin{align}\label{eq:coefficientidentity}
\begin{split}
  &\sum_{k=1}^n\frac{k(k+s)}{(n+1)(n+k+s+1)(n+s+1)}a_{k+s}\overline{b}_k\\
  &\ \ +\sum_{k=n+1}^{\infty}\frac{n(n+s)}
{(k+1)(k+s+1)(n+k+s+1)}a_{k+s}\overline{b}_k=0
\end{split}
\end{align}
for every $n\geqslant 1$.

For fixed $s\geqslant 0$, define
\begin{equation*}\label{eq:kernel}
K_s(n,k)=
\begin{cases}
\displaystyle
\frac{k(k+s)}{(n+1)(n+s+1)(n+k+s+1)},&1\leqslant k\leqslant n,\vspace{2mm}\\[9pt]
\displaystyle
\frac{n(n+s)}{(k+1)(k+s+1)(n+k+s+1)},& k\geqslant n.
\end{cases}
\end{equation*}
Observe that $K_{s}(n, k)=K_s(k, n)$ for all positive integers $n$ and $k$. Moreover,
\eqref{eq:coefficientidentity} can be rewritten as follows:
\begin{equation*}\label{eq:kernel simply}
\sum_{k=1}^{\infty}K_s(n,k)a_{k+s}\overline{b}_k=0.
\end{equation*}
Let $c_k=a_{k+s}\overline{b}_k$, $k=1, 2, \cdots.$ Then  the above equation
 becomes
\begin{equation}\label{eq:rowsystem}
\sum_{k=1}^{\infty}K_s(n,k)c_k=0,
\qquad n\geqslant 1.
\end{equation}

 Recalling that $\sum\limits_{k=1}^{\infty}
\frac{n+s+1}{n+k+s+1}a_{k+s}\overline{b}_k$ is  absolutely convergent, we have
 $\{v_kc_k\}_{k=1}^\infty \in\ell^1$, where $$v_k=\frac1{(k+1)(k+s+1)}.$$
For $k= 1, 2, \cdots$, we define
\begin{equation*}\label{eq:qv}
q_k=k(k+s)(k+1)(k+s+1).
\end{equation*}
Then the kernel $K_s(n, k)$ in (\ref{eq:rowsystem})  can be rewritten as
\begin{equation*}\label{eq:kernelfactorization}
K_s(n,k)=
\frac{q_{\min\{n,k\}}v_nv_k}{n+k+s+1}.
\end{equation*}
It follows that
\begin{align*}
  \left|\frac{K_s(n,k)}{q_n}\right|=\frac{q_{\min\{n,k\}}}{q_n} \frac{ v_nv_k}{n+k+s+1}\leqslant v_nv_k,
\end{align*}
which gives
\begin{equation*}\label{eq:doubleabsolute}
\sum_{n,k=1}^{\infty}
\left|\frac{K_s(n,k)c_k\overline{c}_n}{q_n}\right|
\leqslant\left(\sum_{k=1}^{\infty}|v_kc_k|\right)^2<\infty.
\end{equation*}
Thus we conclude  that the double series $\sum\limits_{n,k=1}^{\infty}
\frac{K_s(n,k)c_k\overline{c}_n}{q_n}$ is absolutely convergent.
It follows  that
$$\sum_{n,k=1}^{\infty}
\frac{K_s(n,k)c_k\overline{c}_n}{q_n}=\sum_{n=1}^{\infty}\sum_{k=1}^{\infty}
\frac{K_s(n,k)c_k\overline{c}_n}{q_n}.$$
Multiplying Equation \eqref{eq:rowsystem} by $\frac{\overline{c}_n}{q_n}$ and summing over $n\geqslant 1$, we obtain
$$\sum_{n=1}^{\infty}\frac{\overline{c}_n}{q_n}\sum_{k=1}^{\infty}
K_s(n,k)c_k=0,$$
which is equivalent to
\begin{equation}\label{eq:fullquadraticzero}
\sum_{n,k=1}^{\infty}\frac{K_s(n,k)c_k\overline{c}_n}{q_n}=0,
\end{equation}
and then we have
\begin{equation}\label{eq:fullquadraticzero1}
\sum_{n,k=1}^{\infty}\frac{K_s(k,n)c_n\overline{c}_k}{q_k}=0.
\end{equation}
Since $K_s(n,k)=K_s(k,n)$ for all $n\geqslant 1$ and $k\geqslant1$,  (\ref{eq:fullquadraticzero1}) can be rewritten as
\begin{equation}\label{eq:fullquadraticzero4}
\sum_{n,k=1}^{\infty}\frac{K_s(n,k)c_k\overline{c}_n}{q_k}=0.
\end{equation}
 Adding Equations (\ref{eq:fullquadraticzero}) and (\ref{eq:fullquadraticzero4}) yields
 \begin{equation}\label{eq:fullquadraticzero2}
 \sum_{n,k=1}^\infty K_s(n,k)\left(\frac{1}{q_k}+\frac{1}{q_n}\right)c_k\overline c_n=0.
 \end{equation}
 Note that
 $$K_s(n,k)\left(\frac{1}{q_k}+\frac{1}{q_n}\right)=\frac{v_kv_n}{n+k+s+1}\left(1+\frac{q_{\min\{n,k\}}}{q_{\max\{n,k\}}}\right).$$ Thus,  Equation (\ref{eq:fullquadraticzero2}) can be rewritten as
  \begin{equation*}\label{eq:fullquadraticzero3}
 \sum_{n,k=1}^\infty \frac{v_kv_n}{n+k+s+1}\left(1+\frac{q_{\min\{n,k\}}}{q_{\max\{n,k\}}}\right)c_k\overline c_n=0.
 \end{equation*}

Next, we are going to show  that $c_k=0$ for all $k\geqslant1.$
Letting $N$ be a positive integer, we begin by analyzing the following quadratic form:
$$S_N=\sum_{k, n=1}^{N}\frac{v_kv_n}{n+k+s+1}\left(1+\frac{q_{\min\{n,k\}}}{q_{\max\{n,k\}}}\right)c_k \overline{c}_n.$$
Using
\begin{equation*}\label{eq:integralrepresentation}
\frac1{n+k+s+1}=\int_0^1r^{n+k+s}\,dr,
\end{equation*}
we obtain
\begin{equation*}\label{eq:finitegram}
\begin{aligned}
S_N&=\sum_{k, n=1}^{N}\frac{v_kv_n}{n+k+s+1}\left(1+\frac{q_{\min\{n,k\}}}{q_{\max\{n,k\}}}\right)c_k\overline c_n\\
&=\sum_{k, n=1}^{N}\frac{v_kv_n}{n+k+s+1}c_k\overline c_n+\sum_{k, n=1}^{N}\frac{v_kv_n}{n+k+s+1}\frac{q_{\min\{n,k\}}}{q_{\max\{n,k\}}}c_k\overline c_n\\
&=\sum_{n,k=1}^{N} \int_{0}^{1} v_n r^n\overline{c}_nv_kr^k c_k r^s dr+\sum_{n,k=1}^{N} \int_{0}^{1}v_nv_k r^{n+k+s}\frac{q_{\min\{n,k\}}}{q_{\max\{n,k\}}}c_k\overline c_ndr\\
&:=\text{I}_1+\text{I}_2.
\end{aligned}
\end{equation*}
For the first term $\text{I}_1$, we have
\begin{align*}
  \text{I}_1&=\int_{0}^{1} \bigg(\sum_{n=1}^{N} v_n r^n\overline{c}_n\bigg)\bigg(\sum_{k=1}^{N}v_kr^k c_k\bigg)r^s dr\\
  &=\int_{0}^{1} \overline{\bigg(\sum_{n=1}^{N} v_n r^nc_n\bigg)}\bigg(\sum_{k=1}^{N}v_kr^k c_k\bigg)r^s dr\\
  &=\int_{0}^{1} \bigg|\sum_{k=1}^{N} v_k r^kc_k\bigg|^2r^s dr\geqslant 0.
\end{align*}
For the second term $\text{I}_2$ in $S_N$, since
\begin{align*}
  \frac{q_{\min\{n,k\}}}{q_{\max\{n,k\}}}=
  \begin{cases}
    \frac{q_n}{q_k},\ \ &n\leqslant k,\vspace{2mm}\\
    \frac{q_k}{q_n},\ \ &n\geqslant k,
  \end{cases}
\end{align*}
an elementary computation gives that
$$\frac{q_{\min\{n,k\}}}{q_{\max\{n,k\}}}=e^{-|\log(q_n)-\log(q_k)|}.$$
Observing that
$$e^{-|t|}=\frac{1}{\pi}\int_{\R}\frac{e^{\mathrm{i}\lambda t}}{1+\lambda^2}d\lambda \ \ \ \ (t\in \mathbb R),$$ we have
\begin{align*}
  e^{-|\log (q_k)-\log (q_n)|}&=\frac{1}{\pi}\int_{\R}\frac{e^{\mathrm{i}t [\log (q_k)-\log (q_n)]}}{1+t^2}dt\\
  &=\frac{1}{\pi}\int_{\R}\frac{e^{\mathrm{i}t \log (q_k)}e^{-\mathrm{i}t\log (q_n)}}{1+t^2}dt.
\end{align*}
It follows that
\begin{align*}
  \text{I}_2&=\frac{1}{\pi}\sum_{n,k=1}^{N} v_nv_k c_k\overline c_n\int_{0}^{1}r^{n+k+s}\int_{\R}\frac{e^{\mathrm{i}t \log(q_k)}e^{-\mathrm{i}t\log (q_n)}}{1+t^2}dtdr\\
  &=\frac{1}{\pi}\sum_{n,k=1}^{N} v_nv_k c_k\overline c_n\int_{\R}\int_{0}^{1}r^{n+k+s}\frac{e^{\mathrm{i}t \log (q_k)}e^{-\mathrm{i}t\log (q_n)}}{1+t^2}drdt\\
  &=\frac{1}{\pi}\int_{\R}\int_{0}^{1}\sum_{n,k=1}^{N} r^{n+k+s}\frac{e^{\mathrm{i}t \log (q_k)}e^{-\mathrm{i}t\log (q_n)}}{1+t^2}v_nv_k c_k\overline c_ndrdt\\
  &=\frac{1}{\pi}\int_{\R}\frac{1}{1+t^2}\int_{0}^{1}\left|\sum_{k=1}^{N} r^ke^{\mathrm{i}t\log (q_k)}v_kc_k\right|^2r^sdrdt\geqslant 0.
\end{align*}
As  $\text{I}_1$ and $\text{I}_2$ are both nonnegative, we arrive at
\begin{align}\label{eq:finitegram}
\begin{split}
S_N=\mathrm{I}_1+\mathrm{I}_2
\geqslant \int_{0}^{1}\bigg|\sum_{k=1}^{N} v_k r^kc_k\bigg|^2r^s dr\geqslant 0.
\end{split}
\end{align}
So far, we  have established that $S_N\geqslant 0$ for all  $N\geqslant 1$.

Define $$F(z)=\sum\limits_{k=1}^{\infty}v_kc_kz^k \ \ \ \ (z\in \mathbb D),$$
where $$v_k=\frac1{(k+1)(k+s+1)}\ \ \ \ \mathrm{and}\ \ \ \ c_k=a_{k+s}\overline{b}_k, \ \ \ \ k=1,2,\cdots.$$
Since $|v_kc_kz^k|\leqslant |v_kc_k|$ for $k\geqslant 1$ on $\D$ and $\sum\limits_{k=1}^{\infty}|v_kc_k|<\infty$, $F$ is analytic on $\D$.
For each  $r\in (0, 1)$, let
$$F_N(r)=\sum\limits_{k=1}^{N}v_kc_kr^k,\ \ \ \  N=1, 2,\cdots.$$
Then we have
\[
|F(r)-F_N(r)|
\leqslant \sum_{k=N+1}^{\infty}|v_kc_kr^{k}|\leqslant \sum_{k=N+1}^{\infty}|v_kc_k|,
\]
which implies  that $\{F_N\}_{N=1}^\infty$ converges uniformly to $F$ on $(0, 1)$.

Using
\eqref{eq:finitegram}, we obtain
\begin{align*}
  0\leqslant\lim_{N\to \infty}\int_{0}^1|F_N(r)|^2r^s dr\leqslant \lim_{N\to \infty}S_N =\sum_{n,k=1}^\infty \frac{v_kv_n}{n+k+s+1}\left(1+\frac{q_{\min\{n,k\}}}{q_{\max\{n,k\}}}\right)c_k\overline c_n=0.
  \end{align*}
  It follows that
  $$\lim_{N\to \infty}\int_{0}^1|F_N(r)|^2r^s dr=0,$$
to obtain
 $$\int_{0}^1|F(r)|^2r^sdr=0$$
 by the uniform convergence.  This implies  that $F=0$  a.e. on $(0,1)$.
  Since  $F$ is continuous on $(0, 1)$,   $F$ must vanish on the whole interval $(0,1)$. However, by the isolation of zeros of analytic functions\textcolor[rgb]{1.00,0.00,0.00}{, }we obtain that
$F$ is identically zero on $\D$.
Therefore,
$$v_kc_k=0,\ \ \ \  k=1,2,\cdots.$$
Since $v_k>0$ for any $k\geqslant 1$,  we have
$$c_k=a_{k+s}\overline{b}_k=0,
\qquad k\geqslant 1   \ \ \mathrm{and} \  \   s\geqslant0,$$
which can be rewritten as
$$
a_p\overline{b}_{q}=0,
\qquad p\geqslant q\geqslant1.
$$
Here, $a_k$ and $b_k$ denote the $k$-th Taylor coefficients of $f$ and $g$, respectively.

Since \(T_f\) and \(T_{\overline{g}}\) may be unbounded on the harmonic Bergman space $L_h^2$, we cannot immediately conclude that \((T_f T_{\overline{g}})^* = T_g T_{\overline{f}}\) on  $L_h^2$. Therefore, to further obtain the relationship between the Taylor coefficients of \(f\) and \(g\), we still need to invoke Lemma \ref{Lconjugation}.

 Let $n\geqslant 1$. Then applying the commutativity relation
$$T_fT_{\overline{g}}=T_{\overline{g}}T_f$$
to the function $\overline{z}^{n}\in L_h^2$, we obtain
$$Q\left[fQ(\overline{z}^n \overline{g})\right]=Q\left[\overline{g}Q(f\overline{z}^{n})\right].$$
Taking complex conjugates on both sides and using Lemma \ref{Lconjugation},
we have
$$
Q\Big[\overline{f}\, \overline{Q(\overline{z}^{n}\overline{g})}\Big]=Q\Big[g\overline{Q(f\overline{z}^{n})}\Big].$$
Using Lemma \ref{Lconjugation} again,  we  obtain that
$$\overline{Q(\overline{z}^{n}\overline{g})} =Q(gz^n)$$
and
$$\overline{Q(f\overline{z}^n)}=Q(\overline{f}z^n),$$
which yields that
$$Q\Big[\overline{f}Q(gz^n)\Big] =Q\Big[gQ(\overline{f}z^n)\Big].$$
This gives
$$T_{\overline{f}}T_g z^n =T_gT_{\overline{f}} z^n,\ \ \ \ n\geqslant 1.$$
Now using the same technique as in the treatment of the case $T_fT_{\overline{g}}=T_{\overline{g}}T_f$, we derive that
$$b_{k+s}\overline{a}_k=0,
\qquad k\geqslant1 \  \ \mathrm{and} \  \ s\geqslant0,
$$
which is equivalent to
$$
a_k\overline{b}_{k+s}=0, \ \ \ \
k\geqslant 1 \  \ \mathrm{and} \  \ s\geqslant0.$$
 Therefore, we obtain that
$$
a_p\overline{b}_{q}=0, \ \ \ \ p\geqslant 1  \ \ \mathrm{and} \ \ q\geqslant1.$$

In order to finish the proof of the theorem, we need to discuss the following cases:

\textbf{Case 1.} For any $p, q\geqslant 1$, both $a_p$ and $b_q$ equal $0$. In this case, we have $f=g=0$ immediately.

\textbf{Case 2.} There exists $m\geqslant 1$ such that $a_m\neq 0$, or there exists $n\geqslant 1$ such that $b_n\neq 0$. In this situation, we have that either $b_i=0$ for all $i\geqslant 1$ or $a_i=0$ for all $i\geqslant 1$. Thus either $g$ or $f$ is constant.

In either case, it can be shown that at least one of $f$ or $g$
 is a constant function. This completes the
proof of Theorem \ref{thm:main}.
\end{proof}

\begin{remark}
It follows from the proof of Theorem \ref{thm:main} that the conclusion can be extended straightforwardly to the standard weighted harmonic Bergman spaces
\(L_h^2(\mathbb D, dA_\alpha)\), where the measure is given by
\[
dA_\alpha(z)=(\alpha+1)(1-|z|^2)^\alpha\,dA(z), \qquad \alpha>-1.
\]
\end{remark}
\vspace{3mm}

\subsection*{Acknowledgment}
This work was supported by the National Natural Science Foundation of China
(12371125)  and  Chongqing Natural Science Foundation (CSTB2024NSCQ-MSX0177).

\end{document}